\documentclass[12pt]{article}
\usepackage{amssymb,amsmath, epsfig}
\include{epsf}

\usepackage{graphicx}
\usepackage{amsfonts}
\usepackage{amscd}
\usepackage[all]{xy}
\advance \textheight by \topskip
\makeatletter

\renewcommand{\theequation}{\thesection.\arabic{equation}}

\def\appendix#1
{
\addtocounter{section}{1}\setcounter{equation}{0}
 \renewcommand{\thesection}{\Alph{section}}
 \section*{Appendix~\thesection\protect\indent \parbox[t]{11.715cm}{#1}}
 \addcontentsline{toc}{section}{Appendix \thesection\ \ \ #1}
}

\def\KK{{\mathbb K}}

\def\CC{{\mathbb C}}
\def\RR{{\mathbb R}}

\newcommand{\eq}{\begin{equation}}
\newcommand{\eqa}{\begin{eqnarray}}
\newcommand{\en}{\end{equation}}
\newcommand{\ena}{\end{eqnarray}}

\newtheorem{theorem}{Theorem}[section]

\newtheorem{corollary}[theorem]{Corollary}

\newtheorem{definition}[theorem]{Definition}
\newtheorem{example}[theorem]{Example}

\newtheorem{lemma}[theorem]{Lemma}

\newtheorem{proposition}[theorem]{Proposition}
\newtheorem{remark}[theorem]{Remark}

\def\qed{\hfill$\square$}
\newenvironment{proof}{\begin{Proof}}{\qed\end{Proof}}

\def\beqa{\begin{eqnarray}}
\def\eeqa{\end{eqnarray}}

\def\I{\mathcal{I}}

\def\KK{{\mathbb K}}

\def\CC{{\mathbb C}}
\def\RR{{\mathbb R}}

\newcommand{\be}{\begin{equation}}
\newcommand{\ee}{\end{equation}}
\newcommand{\beq}{\begin{equation}}
\newcommand{\eeq}{\end{equation}}
\newcommand{\bea}{\begin{eqnarray}}
\newcommand{\eea}{\end{eqnarray}}
\def\beqa{\begin{eqnarray}}
\def\eeqa{\end{eqnarray}}

\def\KK{{\mathbb K}}

\def\CC{{\mathbb C}}
\def\RR{{\mathbb R}}

\begin{document}
\title{
\vskip 1cm
{\bf Recipe theorem for the Tutte polynomial for matroids, renormalization group-like approach} }
\author{
{\sf   G\'erard H. E. Duchamp${}^{a}$\thanks{e-mail: ghed@lipn.univ-paris13.fr
}, }
{\sf   Nguyen Hoang-Nghia${}^{a}$\thanks{e-mail: nguyen.hoang@lipn.univ-paris13.fr
}, }
{\sf   Thomas Krajewski${}^{b}$\thanks{e-mail: thomas.krajewski@cpt.univ-mrs.fr},   
}\\ 
 and
{\sf Adrian Tanasa}${}^{a,c}$\thanks{e-mail: adrian.tanasa@ens-lyon.org}}
\date{\today}
\maketitle

\vskip-1.5cm

\vspace{2truecm}

\begin{abstract}
\noindent 
Using a quantum field theory renormalization group-like differential equation, 
we give a new proof of the recipe theorem for the Tutte polynomial for matroids.
The solution of such an equation is in fact given by some appropriate characters of the 
Hopf algebra of isomorphic classes of matroids, characters which are then related to the 
Tutte polynomial for matroids. 
This Hopf algebraic approach also allows to prove, in a new way, a matroid Tutte polynomial
 convolution formula appearing in W. Kook {\it et. al., J. Comb. Series} {\bf B 76} 
(1999).
\end{abstract}



Keywords: Tutte polynomial for matroids, quantum field theory renormalization-group, 
combinatorial Hopf algebras, Hopf algebras characters.
 

\newpage

\section{Introduction and motivation}
\renewcommand{\theequation}{\thesection.\arabic{equation}}
\setcounter{equation}{0}
\label{introduction}

{\it Combinatorial physics} is a growing field, defining itself 
as the interdisciplinary
domain overlapping between combinatorics and physics. 

One can clearly identify nowadays several aspects of this larger and larger 
overlapping.
Thus, the interference between combinatorics and 
quantum mechanics has been investigated in \cite{bf}, \cite{5} (and references within).
An important interplay area can be 
noticed 
between combinatorics (bijective, enumerative and so on) on one hand 
 and statistical physics and integrable combinatorics on another hand 
(see, for example, 
\cite{dflast} for a recent review on this topic).

\medskip

Our paper deals with yet another aspect of this interdisciplinary field of research,
namely the interference between combinatorics and quantum field theory (commutative or non-commutative).
This type of interference has already appeared in the pioneering work 
of Alain Connes and Dirk Kreimer, who have defined a 
Hopf algebra encoding in an elegant way the combinatorics of the process of perturbative renormalization 
in commutative quantum field theory.
This structure has then been generalized for noncommutative quantum field theory 
\cite{fab}, \cite{io-dirk} and for spin-foam quantum gravity models 
\cite{mar}, \cite{io-sf}. For general reviews of various
interferences between algebraic and analytic (and not only) combinatorics and quantum field theory 
in general, we invite the interested reader to consult \cite{io-hdr} and \cite{io-slc}.

In this paper we 
consider appropriate characters of the the Hopf algebra of isomorphic classes of matroids, 
algebra defined 
 (as a particularization of a more general construction of incidence Hopf algebra) 
in 
\cite{s} (and then extensively studied in \cite{cs}). 
We also show that these characters are related to the Tutte polynomial for matroids 
and we then 
use a quantum field theory renormalization group-like 
differential equation to prove the universality of the Tutte polynomial for matroids.
More precisely, we show that a solution of such an equation is given by 
the characters that we have defined. As 
a by-product of our Hopf algebraic approach, 
we give a new proof of a convolution formula for 
the Tutte polynomial for matroids, formula exhibited in \cite{reiner}.

\medskip

The paper is structured as follows. In the following section we introduce 
the renormalization group equation in quantum field theory and we then recall 
some useful notions related to matroids and to the Tutte polynomial for matroids.
Finally, we give the definition of the Hopf algebra of 
isomorphic classes of matroids.
In the third section we define the Hopf algebra characters that will be used in the sequel.
In the following section, we use this construction to give our new proof of the 
convolution formula for the Tutte polynomial for matroids mentioned above.
The fifth section is dedicated to our main result, the proof of the universality of the 
Tutte polynomial for matroids.
The last section presents some concluding remarks and perspectives for future work.

\section{Quantum field theory and matroid reminders}
\renewcommand{\theequation}{\thesection.\arabic{equation}}
\setcounter{equation}{0}

In this section we first briefly recall some quantum field theory notions, 
namely the renormalization group differential equation.
We then give the definition of matroids, of the associated 
Tutte polynomial as well as some further properties which will be useful 
to prove the results of this paper.
Finally, the Hopf algebra of isomorphic classes of matroids is given.

\subsection{Quantum field theory, the renormalization group}
\renewcommand{\theequation}{\thesection.\arabic{equation}}
\setcounter{equation}{0}

A QFT model (for a general introduction 
to QFT  see for example the books \cite{book-ZJ} or \cite{kleiner}) 
is defined by means of a functional integral of the exponential of an 
{\it action} $S$ which, from a mathematical point of view, is a functional
 of the {\it fields} of the model. For the $\Phi^4$ scalar model - the simplest QFT model - 
there is only one type of
field, which we denote by $\Phi (x)$. From a mathematical point of view,
for an Euclidean QFT scalar model,
one has $\Phi:\RR^D\to \KK$, 
where $D$ is usually taken equal to $4$ (the dimension of the space) 
and $\KK\in\{\RR,\CC\}$ (real, respectively complex fields).

The quantities computed in QFT are generally divergent. One thus has to consider a 
real, positive, {\it cut-off}
 $\Lambda$ - the flowing parameter.
 This leads to a family of cut-off dependent actions, family denoted by
$S_\Lambda$. The derivation $\Lambda\frac{\partial S_\Lambda}{\partial\Lambda}$ gives the 
{\it renormalization group equation}.

The quadratic part of the action - the {\it propagator} of the model - can be written in the following 
way
\beqa
\label{propa}
C_{\Lambda,\Lambda_0}(p,q)=\delta(p-q)\int_{\frac{1}{\Lambda_0}}^{\frac{1}{\Lambda}}
d\alpha e^{-\alpha p^2},
\eeqa
with $p$ and $q$ living in the Fourier transformed space $\RR^D$ and 
$\Lambda_0$ a second real, positive cut-off. In perturbative QFT, one has to 
consider {\it Feynman graphs}, and to associate to each such a graph a 
{\it Feynman integral}  
(further related to quantities actually measured in physical experiments).
The contribution of an edge of such a Feynman graph to its associated Feynman integral 
is given by an integral such as \eqref{propa}.

One can then get (see \cite{pol}) 
the Polchinski flow equation
\beqa
\label{eq:pol}
\Lambda \frac{\partial S_\Lambda}{\partial\Lambda}=
\int_{\RR^{2D}} \frac 12 d^D p d^D q
\Lambda\frac{\partial C_{\Lambda,\Lambda_0}}{\partial \Lambda}
\left(
\frac{\delta^2S}{\delta \tilde \Phi (p)\delta \tilde \Phi (q)}
-\frac{\delta S}{\delta \tilde \Phi (p)} \frac{\delta S}{\delta \tilde \Phi (q)}
\right),
\eeqa
where $\tilde \Phi$ represents the Fourier transform of the function $\Phi$. 
The first term in the right hand side (rhs) of the equation above corresponds 
to the derivation of a propagator associated to a bridge 
in the respective Feynman graph.
The second term corresponds to an edge which is not a bridge and is 
part of some circuit in the graph. One can see this diagrammatically in Fig. \ref{fig:pol}.
\begin{figure}[h]
\centerline{\includegraphics[width=12cm]{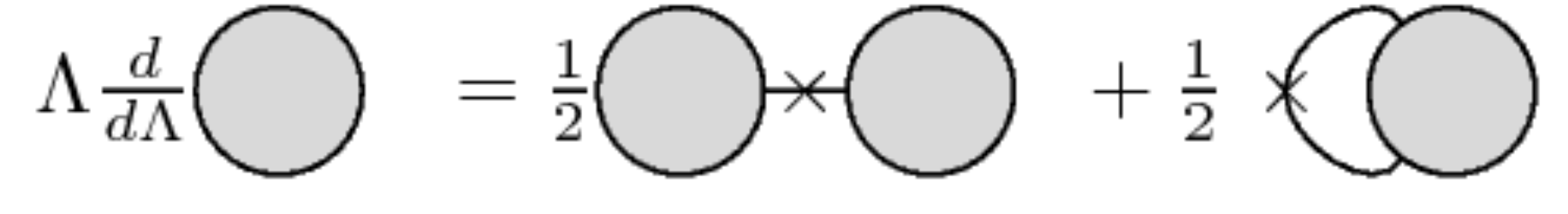}   }
\caption{Diagrammatic representation of the flow equation.}
\label{fig:pol}
\end{figure}

This equation can then be used to prove perturbative renormalizability in QFT.
Let us also stress here, that an equation of this type is also used 
to prove a result of E. M. Wright which expresses the generating function 
of connected graphs under certain conditions (fixed excess). 
To get this generating 
functional (see, for example, Proposition $II.6$ the book \cite{book-fs}), 
one needs to consider contributions of two types of edges (first contribution
when  the 
edge is a bridge and a second one when not - see again Fig. \ref{fig:pol}).


\subsection{Matroids and the Tutte polynomial for matroids}

In this subsection we recall the definition and some properties of the
 Tutte polynomial for matroids as well as of the matroid Hopf algebra 
defined in \cite{s}.

Following the book \cite{oxley}, one has the following definitions:

\begin{definition}
A {\bf matroid} M is a pair $(E, \mathcal{I})$ consisting of a finite set E 
and a collection of subsets of E satisfying the following set of axioms:
$\mathcal{I}$ is non-empty,
every subset of every member of $\mathcal{I}$ is also in $\mathcal{I}$
and, finally,
if $X$ and $Y$ are in $\mathcal{I}$ and $|X| = |Y | + 1$, then there is an element 
$x$ in $X - Y$ such that $Y \cup \{x\}$ is in $\mathcal{I}$.
The set $E$ is the ground set of the matroid and the members of 
$\mathcal{I}$ are the independent sets of the matroid.
\end{definition}

One can define a matroid on the edge set of any graph - {\bf graphic matroid}. The reciproque 
does not hold (not every matroid is a graphic matroid).

Let $E$ be an $n-$element set and let $\I$ be the collection of subsets 
of $E$ with at most $r$ elements, $0\le r\le n$. One can check that $(E,\I)$ 
is a matroid; it is called the {\bf uniform matroid} $U_{r,n}$.

\begin{remark}
\label{remarca}
 If one takes $n=1$, there are only two matroids, namely $U_{0,1}$ and $U_{1,1}$ and both of these 
matroids are graphic matroids. The graphs these two 
matroids correspond to are the graphs with one edge of Fig. \ref{graf1} and Fig. \ref{graf2}.
\begin{figure}[!ht]
\centerline{\includegraphics[scale=0.5]{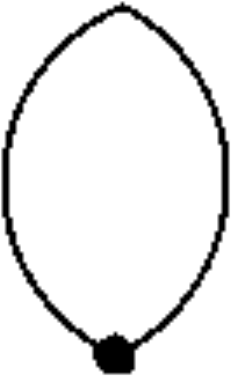}}
\caption{The graph corresponding to the matroid $U_{0,1}$.}
\label{graf1}
\end{figure}
\begin{figure}[!ht]
\centerline{\includegraphics[scale=0.5]{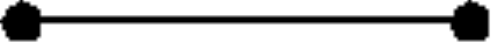}}
\caption{The graph corresponding to the matroid $U_{1,1}$.}
\label{graf2}
\end{figure}
In the first case, the edge is a loop (in graph theoretical terminology) or a tadpole 
(in QFT language). In the second case, the edge represents a bridge (in graph theoretical 
terminology) or a $1$-particle-reducible line (in QFT terminology) - the number of connected 
components of the graphs increases by $1$ if one deletes the respective edge. 
\end{remark}

\begin{definition}
 Maximal independent sets of a matroid are called bases. 
The collection of minimal dependent sets of a matroid are called circuits.
\end{definition}

Let $M=(E,\I)$ be a matroid and let $\cal B=\{ B\}$ be the collection of bases of $M$. 
Let ${\cal B}^\star = \{E - B: B \in {\cal B} \}$. 
Then ${\cal B}^\star$ is the collection of bases of a matroid $M^\star$ on E. 
The matroid $M^*$ is called the dual of $M$.

\begin{definition}
\label{def-rank}
Let $M=(E,\I)$ be a matroid. 
The {\bf rank} $r(A)$ of $A \subset E$ is defined as the cardinal of a maximal independent set in $A$.
\begin{equation}\label{eq:rankfunc}
r(A) = max\{|B| \mbox{  s.t.  } B \in \I, B \subset A\}\ .
\end{equation}
\end{definition}

\begin{definition}
\label{def-nullity}
Let $M=(E,\mathcal{I})$ be a matroid with a ground set $E$. 
The {\bf nullity} function is given by  
\begin{equation}
\label{defnul}
n(M) = |E|- r(M)\ .
\end{equation}
\end{definition}

\begin{definition}
\label{def-loop}
Let $M=(E,\I)$ be a matroid.
The element $e \in E$. 
 is a \textbf{loop} iff $\{e\}$ is the circuit.
\end{definition}

\begin{definition}
\label{def-coloop}
Let $M=(E,\I)$ be a matroid. 
The element $e\in E$ is a \textbf{coloop} iff, for any basis $B$, 
$e\in B$ .
\end{definition}

Let us now define two basic operations on matroids.
Let $M$ be a matroid $(E,\I$) and $T$ be a subset of $E$. 
Let $\I'=\{I\subseteq E-T: I \in \I\}$. 
One can check that $(E-T,\I')$ is a matroid. 
We denote this matroid by $M\backslash T$ - the {\bf deletion} of $T$ from $M$.
The {\bf contraction} of $T$ from $M$, $M/T$, is given by the formula: 
$M/T=(M^\star\backslash T)^\star$.

Let us also recall the following results:

\begin{lemma}\label{lm:res-del}
Let $M$ be a matroid $(E,\I$) and $T$ be a subset of $E$. One has \begin{equation}
M|_T = M\backslash_{E-T} \ .
\end{equation}
\end{lemma}

\begin{lemma}
\label{lm:coloop}
If $e$ is either a coloop or a loop of a matroid $M=(E,\I)$, then $M/e = M\backslash e$.
\end{lemma}

\begin{lemma}
\label{lm:rankcontra}
Let $M=(E,\I)$ be a matroid and $T\subseteq E$, then, for all $X \subseteq E-T$, \begin{equation}
r_{M/T}(X) = r_M(X\cup T) - r_M(T)\ .
\end{equation}
\end{lemma}

\medskip
Let us now define the Tutte polynomial for matroids:

\begin{definition}
Let $M=(E,\mathcal{I})$ be a matroid. The {\bf Tutte polynomial} is given by the following formula:
\begin{equation}
\label{def-tutte}
T_M(x,y) = \sum_{A \subseteq E} (x-1)^{r(E) - r(A)} (y-1)^{n(A)}.
\end{equation}
The sum is computed over all subset of the matroid's ground set. 
\end{definition}

\begin{example}
Let $U_{k,n}$ be a uniform matroid, $0 \leq k \leq n$. The Tutte polynomial of this matroid is given by
\beqa
T_{U_{k,n}}(x,y) = \sum_{i=0}^k \binom{n}{i}(x-1)^{k-i} + \sum_{i=k+1}^n \binom{n}{i} (y-1)^{i-k}.
\eeqa
\end{example}

It is worth stressing here that one can define the dual of any matroid; this is not the case for graphs, 
where only the dual of planar graph can be defined.

Let us recall, from \cite{bo} that
\beqa
\label{reltdual}
T_M (x,y)=T_{M^\star}(y,x).
\eeqa

\subsection{Hopf algebra}

\begin{definition}
Let $M_1$ and $M_2$ be the matroids $(E_1,\I_1)$ and $(E_2,\I_2)$ where $E_1$ and $E_2$ are disjoint. Let \[M_1 \oplus M_2 = (E_1 \cup E_2,\{I_1 \cup I_2: I_1 \in \I_1, I_2 \in \I_2\})\ .\]Then $M_1 \oplus M_2$ is a matroid. This matroid is called the direct sum of $M_1$ and $M_2$.
\end{definition}

In \cite{s}, 
as a particularization of a more general construction of incidence Hopf algebras, the following result 
was proved:

\begin{proposition}
\label{prop-cs}
If $\mathcal{M}$ is a minor-closed family of matroids (and if we denote by $\widetilde{\mathcal{M}}$ the set of isomorphism classes of it),  
then $k(\widetilde{\mathcal{M}})$ is a coalgebra, 
with coproduct $\Delta$ and counit $\epsilon$ 
determined by 
\beqa
\label{defc}
\Delta(M) = \sum_{A\subseteq E} M|A \otimes M/A
\eeqa
and respectively by 
$\epsilon(M) = \begin{cases} 1, \mbox{ if } E = \emptyset, \\ 0 \mbox{ otherwise ,} \end{cases}$ for all $M = (E,\mathcal{I}) \in \mathcal{M}$. If, furthermore, the family $\mathcal{M}$ is closed under formation of direct sums, then $k(\widetilde{\mathcal{M}})$ is a Hopf algebra, with product induced by direct sum.
\end{proposition}

We refer to this Hopf algebra as to the matroid Hopf algebra. We follow \cite{cs}
and, by a slight abuse of notation, we denote in the same way a matroid and its isomorphic class,
since the distinction will be clear from the context (as it is already in 
Proposition \ref{prop-cs}).

We denote the unit of this Hopf algebra by $\mathbf{1}$ (the empty matroid, or $U_{0,0}$).

\begin{example}(Example 2.4 of \cite{cs})
Let $M = U_{k,n}$ be a uniform matroid with rank $k$. Its coproduct is given by
\[\Delta(U_{k,n}) = \sum_{i=0}^k \binom{n}{i} U_{i,i}\otimes U_{k-i,n-i} + \sum_{i=k+1}^n \binom{n}{i} U_{k,i}\otimes U_{0,n-i}\ .\]
\end{example}

%

\section{Hopf algebra characters}
\renewcommand{\theequation}{\thesection.\arabic{equation}}
\setcounter{equation}{0}

Let us give the following definitions:

\begin{definition}
Let $f,g$ be two mappings in $Hom(\mathcal{M},\mathcal{M})$. 
The convolution product of $f$ and $g$ is given by the following formula:
\begin{equation}
\label{def-conv}
f\ast g = (f\otimes g)\circ \Delta \ .
\end{equation}
\end{definition}

\begin{definition}
A matroid Hopf algebra {\bf character} $f$ is an algebra morphism 
from 
the matroid Hopf algebra
into a fixed commutative ring $\mathbb{K}$, i.e. 
such that 
\beqa
f(M_1\oplus M_2) = f(M_1)f(M_2),\ \  f(\mathbf{1}) =1_\KK.
\eeqa
\end{definition}

\begin{definition}
A matroid Hopf algebra {\bf infinitesimal 
character} $g$ is a linear morphism from 
the matroid Hopf algebra 
into a fixed commutative ring $\KK$, 
such that 
\beqa
g(M_1\oplus M_2) = g(M_1)\epsilon(M_2)+\epsilon(M_1)g(M_2).
\eeqa
\end{definition}

Since we work in a Hopf algebra where the non-trivial part of the 
coproduct is nilpotent, we can also define an exponential map by
the following expression 
\beqa
\label{def-expstar}
\mbox{exp}_\ast(\delta)=\epsilon+\delta+\frac 12 \delta\ast\delta +\ldots
\eeqa
where $\delta$ is an infinitesimal character.

As already stated above (see Remark \ref{remarca}),  
there are only two matroids with unit cardinal ground set,
$U_{0,1}$ and $U_{1,1}$. We now define two 
applications $\delta_{\mathrm{loop}}$ and $\delta_{\mathrm{coloop}}$.


\beqa
\label{def-dloop}
\delta_{\mathrm{loop}} (M) = \begin{cases}
1_\KK \mbox{ if } M 
= U_{0,1},\\
0_\KK \mbox{ otherwise }.
\end{cases}
\eeqa

\beqa
\label{def-dtree}
\delta_{\mathrm{coloop}} (M) = \begin{cases}
1_\KK \mbox{ if } M 
= U_{1,1},\\
0_\KK \mbox{ otherwise }.
\end{cases}
\eeqa

One can directly check that these applications are {\it infinitesimal characters} of the matroid 
Hopf algebra defined above. 

\medskip


We now define the following application:

\beqa
\label{def-alpha}
\alpha(x,y,s,M) := \mbox{exp}_\ast s\{\delta_{\mathrm{coloop}}
+(y-1)\delta_{\mathrm{loop}}\}\ast \mbox{exp}_\ast s\{(x-1)\delta_{\mathrm{coloop}}+\delta_{\mathrm{loop}}\}
(M).
\eeqa


\begin{example}
Let $U_{k,n}$ be a uniform matroid, $0 \leq k \leq n$. One has
\beqa
\alpha(x,y,s,U_{k,n}) = \sum_{i=0}^k \binom{n}{i}s^n(x-1)^{k-i} + \sum_{i=k+1}^n \binom{n}{i} s^n(y-1)^{i-k} =s^nT_{U_{k,n}}(x,y).
\eeqa
\end{example}

\medskip
One then has:

\begin{proposition}
 The application \eqref{def-alpha} is a character.
\end{proposition}
\begin{proof}
The proof can be done by a direct check. On a more general 
basis, this is a consequence of the fact that 
$\delta_{\mathrm{loop}}$ and $\delta_{\mathrm{coloop}}$ are infinitesimal characters
and the space of infinitesimal characters is a vector space; 
thus 
$s\{\delta_{\mathrm{coloop}}
+(y-1)\delta_{\mathrm{loop}}\}$ and 
$s\{(x-1)\delta_{\mathrm{coloop}}+\delta_{\mathrm{loop}}\}$ 
are
infinitesimal characters.
Since, $\mbox{exp}_\ast (h)$ is a character when $h$ is an infinitesimal character
and since the convolution of two characters is a character,
one gets that $\alpha$ is a character.
\end{proof}

\section{Proof of a Tutte polynomial convolution formula}
\renewcommand{\theequation}{\thesection.\arabic{equation}}
\setcounter{equation}{0}

Let $M=(E,\mathcal{I})$ be a matroid. 

One then has:

\begin{lemma}
\label{lema-help}
 Let $M=(E,\mathcal{I})$ be a matroid. One has 
\begin{equation}
\label{help}
\mbox{exp}_\ast\{a\delta_{\mathrm{coloop}}+b\delta_{\mathrm{loop}}\}(M) = a^{r(M)}b^{n(M)}.
\end{equation} 
\end{lemma}
\begin{proof}
Using the definition \eqref{def-expstar}, the lhs of the identity \eqref{help} above
reads:
\begin{equation}
 \left(\sum_{k=0}^\infty \frac{(a\delta_{\mathrm{coloop}}+b\delta_{\mathrm{loop}})^{\ast k}}{k!}\right)(M).
\end{equation}
All the terms in the sum above vanish, except the one for whom 
$k$ is equal to $|E|$. Using the definition \eqref{def-conv} of the convolution product,
this terms writes
\beqa
\frac{1}{k!}\left( \sum_{i=0}^k a^{k-i}b^{i}
\sum_{\substack{i_1+ \dots + i_n =k - i \\ j_1+ \dots + j_m = i}}
\delta_{\mathrm{coloop}}^{\otimes(i_1)}\otimes\delta_{\mathrm{loop}}^{\otimes(j_1)}
\otimes\dots\otimes \delta_{\mathrm{coloop}}^{\otimes(i_n)}
\otimes\delta_{\mathrm{loop}}^{\otimes(j_m)}\right) 
\left(\sum_{ (i)} M^{(1)}\otimes\dots\otimes M^{(k)}\right),
\eeqa
where we have used the notation $\Delta^{(k-1)}(M)=\sum_{ (i)} M^{(1)}\otimes\dots\otimes M^{(k)}$.
Using the definitions \eqref{def-dloop} and respectively \eqref{def-dtree} 
of the infinitesimal characters
$\delta_{\mathrm{loop}}$ and respectively $\delta_{\mathrm{coloop}}$, implies that
the submatroids $M^{(j)}$ ($j=1,\ldots,k$) are equal to $U_{0,1}$ or $U_{1,1}$.

Using the definition of the nullity and of the rank of a matroid concludes the proof.
\end{proof}

\begin{example}
Let us illustrate Lemma \ref{lema-help} for the uniform matroid $U_{k,n}$. One has $r(U_{k,n}) = k$ and $n(U_{n,k}) = n - k$.
We now use the definitions \ref{def-dloop} and \ref{def-dtree} of $\delta_{loop}$ and $\delta_{coloop}$ to work out the lhs of identity \ref{help}.
\begin{align}
{exp}_*\{a\delta_{\mathrm{coloop}}+b\delta_{\mathrm{loop}}\}(U_{k,n}) & = \frac{1}{n!} a^{k}b^{n-k} \delta_{\mathrm{coloop}}^{\otimes k}\otimes \delta_{\mathrm{loop}}^{\otimes (n-k)} \left(\binom{n}{n-1}\dots \binom{2}{1} U_{1,1}^{\otimes k}\otimes U_{0,1}^{\otimes (n-k)} \right) \notag\\
& = a^kb^{n-k}.
\end{align}
\end{example}

\medskip

One has:
\beqa
\label{pre-reiner}
\alpha(x,y,s,M) &=& 
\mbox{exp}_{\ast}
\left(s(\delta_{\mathrm{coloop}}+(y-1)\delta_{\mathrm{loop}})\right)\ast\mbox{exp}_{\ast}\left(s(-\delta_{\mathrm{coloop}}+\delta_{\mathrm{loop}})\right)
\nonumber\\
&\ast&\mbox{exp}_{\ast}\left(s(\delta_{\mathrm{coloop}}-\delta_{\mathrm{loop}})\right)\ast
\mbox{exp}_{\ast}\left(s((x-1)\delta_{\mathrm{coloop}}+\delta_{\mathrm{loop}})\right).
\eeqa 

\begin{proposition}
\label{alpha-tutte}
Let $M=(E,\mathcal{I})$ be a matroid. The character $\alpha$ is related to the Tutte polynomial of matroids by the 
following identity:
\begin{equation}
\alpha(x,y,s,M) = s^{|E|}T_M(x,y).
\end{equation} 
\end{proposition}
\begin{proof}
 Using the definition \eqref{def-conv} of the convolution product in the definition 
\eqref{def-alpha} of the character $\alpha$, one has the following identity:
\begin{align}
\label{eq:dev2mat} 
\alpha(x,y,s,M) = \sum_{A \subseteq E} \mbox{exp}_*s\{\delta_{\mathrm{coloop}}+(y-1)\delta_{\mathrm{loop}}\}(M|_A)\ 
\mbox{exp}_*s\{(x-1)\delta_{\mathrm{coloop}}+\delta_{\mathrm{loop}}\}(M/A).
\end{align}
We can now apply Lemma \ref{lema-help} on each of the two terms in the rhs of equation 
\eqref{eq:dev2mat} above. This leads to the result.
\end{proof}

\medskip

Using \eqref{reltdual} and the Proposition \ref{alpha-tutte}, one has:

\begin{corollary}
 One has:
\beqa
\alpha(x,y,s,M) = \alpha(y,x,s,M^\star).
\eeqa
\end{corollary}

The Proposition \ref{alpha-tutte} allows to give a different proof of a matroid Tutte polynomial
convolution identity, which was shown in \cite{reiner}.
One has:

\begin{corollary}
(Theorem $1$ of \cite{reiner})
 The Tutte polynomial satisfies
\begin{equation}
 T_M(x,y)=\sum_{A\subset E} T_{M|A}(0,y) T_{M/A}(x,0).
\end{equation}
\end{corollary}
\begin{proof}
Taking $s=1$, this is as a direct consequence of 
identity \eqref{pre-reiner}, 
and of Proposition \ref{alpha-tutte}. 
\end{proof}

\section{The recipe theorem}
\renewcommand{\theequation}{\thesection.\arabic{equation}}
\setcounter{equation}{0}

Let us define an application 
\begin{equation}
\label{defphi}
\varphi_{a,b} (M) \longmapsto a^{r(M)}b^{n(M)}M \ .
\end{equation} 

\begin{lemma}
\label{lemamor}
The application $\varphi_{a,b}$ is a bialgebra automorphism.
\end{lemma}
\begin{proof}
One can directly check that the application $\varphi_{a,b}$ is an algebra automorphism.
Let us now check that this application is also a coalgebra automorphism. Using 
Lemma \ref{lm:res-del} and 
Lemma \ref{lm:rankcontra}, 
\beqa
\label{ranks}
r(M|_T) + r(M/T) = r(M).
\eeqa
Thus, using  the definitions of the application $\varphi_{a,b}$ of the matroid coproduct, one has:
\begin{equation}
 \Delta\circ \varphi_{a,b}(M) =
\sum_{T\subseteq E} (a^{r(M|_T)}b^{n(M|_T)}M|_T) \otimes (a^{r(M/_T)}b^{n(M/_T)}M/_T).
\end{equation}
Using again the definition of the application $\varphi_{a,b}$ leads to
\begin{equation}
 \Delta\circ \varphi_{a,b}(M) =(\varphi_{a,b}\otimes \varphi_{a,b})\circ \Delta (M),
\end{equation}
which concludes the proof.
\end{proof}

\medskip

Let us now define:
\beqa
[f,g]_\ast :=f\ast g - g\ast f.
\eeqa
Using the definition \eqref{def-alpha} of the Hopf algebra character $\alpha$, 
one can directly prove the following result:

\begin{proposition}
 \label{propalpha}
The character $\alpha$ is the solution of the differential equation:
 \begin{equation}
\label{diffeqalpha}
\frac{d\alpha}{ds}(M) = x\alpha \ast \delta_{\mathrm{coloop}} + y\delta_{\mathrm{loop}}\ast\alpha + 
\left[\delta_{\mathrm{coloop}},\alpha\right]_\ast - \left[\delta_{\mathrm{loop}},\alpha\right]_\ast(M).
\end{equation}
\end{proposition}

It is the fact that the matroid Tutte polynomial is a solution of the differential equation 
\eqref{diffeqalpha} that will be used now to prove the 
universality of the matroid Tutte polynomial. In order to do that, we take a
four-variable matroid polynomial $Q_M(x,y,a,b)$ satisfying 
a multiplicative law and which has the following properties: 
\begin{itemize}
 \item if $e$ is a coloop, then
\begin{equation}
Q_M(x,y,a,b) = xQ_{M\backslash e}(x,y,a,b)\ ,
\end{equation}
\item  if $e$ is a loop, then
\begin{equation}
 Q_M(x,y,a,b) = yQ_{M/e}(x,y,a,b)
\end{equation}
\item if $e$ is a nonseparating point, then
\beqa
Q_M(x,y,a,b) = aQ_{M\backslash e}(x,y,a,b)+bQ_{M/e}(x,y,a,b).
\eeqa
\end{itemize}

\begin{remark}
 Note that, when one deals with the same problem in the case of graphs, 
a supplementary multiplicative condition for the case of one-point joint of two graphs 
({\it i. e.} identifying a vertex of the first graph and a vertex of the second graph 
  into a single vertex of the resulting graph) 
is required (see, for example, \cite{em} or \cite{sokal}).
\end{remark}

We now 
define the application:
\beqa
\label{defbeta}
\beta(x,y,a,b,s,M):=s^{|E|}Q_M(x,y,a,b).
\eeqa
One then directly checks (using the definition \eqref{defbeta} above and 
the multiplicative property of the polynomial $Q$)  
that 
this application is again a matroid Hopf algebra character.

\begin{proposition}
\label{propbeta}
The character \eqref{defbeta} satisfies the following differential equation: 
\begin{equation}
\label{diffeqbeta}
\frac{d\beta}{ds} (M) = \left( x\beta \ast \delta_{\mathrm{coloop}} + y\delta_{\mathrm{loop}}\ast \beta + 
b [\delta_{\mathrm{coloop}},\beta]_\ast - a [\delta_{\mathrm{loop}},\beta]_\ast\right) (M).
\end{equation}
\end{proposition}
\begin{proof}
Applying the definition \eqref{def-conv} of the convolution product,
the rhs of equation \eqref{diffeqbeta} above writes
 \begin{align}
\label{eq:betamat}
& = (x-b)\sum_{A\subseteq E} \beta(M|_A) \delta_{\mathrm{coloop}}(M/_A) + (y-a)\sum_{A\subseteq E} \delta_{\mathrm{loop}}(M|_A)\beta(M/_A)\notag\\
& + b \sum_{A\subseteq E} \delta_{\mathrm{coloop}}(M|_A)\beta(M/_A) + a \sum_{A\subseteq E} \beta(M|_A) \delta_{\mathrm{loop}}(M/_A).
\end{align}
Using the definitions \eqref{def-dloop} and respectively 
\eqref{def-dtree} of the infinitesimal characters
$\delta_{\mathrm{loop}}$ and respectively $\delta_{\mathrm{coloop}}$, constraints the sums on the subsets $A$ above. 
The rhs of \eqref{diffeqbeta}
becomes:
\beqa
& (x-b)\sum_{A, M/_A = U_{1,1}} \beta(M|_A) + (y-a)\sum_{A, M|_A= U_{0,1}} \beta(M/_A)\notag\\
& + b \sum_{A, M|_A = U_{1,1}} \beta(M/_A) + a \sum_{A, M/_A = U_{0,1}}\beta(M|_A)
\eeqa
We now apply the definition of the Hopf algebra character $\beta$; one obtains:
\beqa
\label{int1}
&s^{|E|-1}[(x-b)\sum_{A, M/_A = U_{1,1}} Q(x,y,a,b,M|_A) + 
(y-a)\sum_{A, M|_A= U_{0,1}} Q(x,y,a,b,M/_A)\notag\\
& + b \sum_{A, M|_A = U_{1,1}} Q(x,y,a,b,M/_A) + a \sum_{A, M/_A = U_{0,1}}Q(x,y,a,b,M|_A)].
\eeqa
We can now directly analyze the four particular cases $M/_A=U_{1,1}$, $M/_A=U_{0,1}$, 
$M|_A=U_{1,1}$ and $M|_A=U_{0,1}$:
\begin{itemize}
 \item 
If 
 $M/_A=U_{1,1}$, we can denote the ground set of $M/_A$ by $\{ e \}$. Note that $e$ is a coloop.
From the Lemma \ref{lm:res-del}, one has $M|_A = M\backslash_{E-A}=M\backslash e$.
One then has 
$Q(x,y,a,b,M) = xQ(x,y,a,b,M|_A)$.

\item
If $M|_A= U_{0,1}$, then $A=\{e\}$ and $e$ is a loop of $M$. Thus, one has 
$Q(x,y,a,b,M) = yQ(x,y,a,b,M/_A)$

\item If $M|_A = U_{1,1}$, then $A=\{e\}$. 
One has to distinguish between two subcases:
\begin{itemize}
\item $e$ is a coloop of $M$. Then, by Lemma \ref{lm:coloop}, 
$M/e = M\backslash e$. Thus, one has $Q(x,y,a,b,M) = xQ(x,y,a,b,M|_A)$.
\item $e$ is a nonseparating point of $M$. 
\end{itemize}

\item If $M/_A = U_{0,1}$, one can denote the ground set of $M/_A$ by 
$ \{e\}$. 
There are again two subcases to be considered:
\begin{itemize}
\item $e$ is a loop of $M$, one has that $M|_A = M\backslash_{(E-A)} = M\backslash_{\{e\}} = M/e$. 
Then one has $Q(x,y,a,b,M) = yQ(x,y,a,b,M|_A)$.
\item $e$ is a nonseparating point of $M$, then one has $M|_A = M\backslash_{(E-A)} = M\backslash_{\{e\}}$
\end{itemize}
\end{itemize}
We now insert all of this in equation \eqref{int1}; this leads to three types of sums over 
some element $e$ of the ground set $E$, $e$ being a loop, a coloop or a 
nonseparating point:
\begin{align}
s^{|E|-1}[\sum_{e \in E: e \mbox{\tiny is a coloop}} Q(x,y,a,b,M) + \sum_{e\in E: e \mbox{\tiny is a loop}} Q(x,y,a,b,M) + \sum_{e \in E: e \mbox{\tiny is a regular element}}Q(x,y,a,b,M)] \notag\\
\end{align}
This rewrites as
\begin{align}
|E|s^{|E|-1}Q(x,y,a,b,M) = \frac{d\beta}{ds}(M),
\end{align}
which completes the proof.
\end{proof}

\medskip

We can now state the {\it main result} of this paper, the recipe theorem
specifying how to recover the matroid polynomial $Q$ as 
an evaluation of the Tutte polynomial $T_M$:

\begin{theorem}
One has:
\begin{equation}
 \label{eqrecipe}
Q(x,y,a,b,M) = a^{n(M)}b^{r(M)}T_M(\frac{x}{b},\frac{y}{a}).
\end{equation}
\end{theorem}
\begin{proof}
 The proof is a direct consequence of Propositions \ref{alpha-tutte}, \ref{propalpha} and 
\ref{propbeta} and of Lemma \ref{lemamor}. This comes from the fact that 
one can apply the automorphism $\phi$ defined in \eqref{defphi}
to the differential equation 
\eqref{diffeqbeta}. One then obtains the differential equation 
\eqref{diffeqalpha} with modified parameters $x/b$ and $y/a$. 
Finally, the solution of this differential equation is (trivially) related to 
the matroid Tutte polynomial $T_M$ (see Proposition \ref{alpha-tutte})
and this concludes the proof. 
\end{proof}

\bigskip

Let us end this section by stating that all the results obtained in this paper 
naturally hold for graphs (instead of matroids), since graphs are a particular class of 
matroids (the graphic matroids, see subsection $2.2$).
We have thus given here the proofs of the graph results conjectured in \cite{km}. 

\section{Concluding remarks and perspectives}
\renewcommand{\theequation}{\thesection.\arabic{equation}}
\setcounter{equation}{0}

We have  used in this paper a quantum field theory renormalization group-like equation 
to prove the universality of the Tutte polynomial for matroids.
Moreover, we gave a new proof of the convolution identity established in 
\cite{reiner} by W. Kook {\it et. al}.

Let us emphasize that the Hopf algebra coproduct \eqref{defc} 
used in this paper is a so-called type $I$ coproduct, namely 
a coproduct using a {\it selection-quotient} rule.
Examples of such coproducts are the 
Connes-Kreimer coproduct for commutative quantum field theory  Feynman 
graphs \cite{ck}, its generalization to 
non-commutative quantum field theory \cite{fab, io-dirk} and so on.
As it was already noticed in \cite{io-hdr} or \cite{nguyen}, 
this type of rule is fundamentally different of 
the one used to define a so-called type $II$ coproduct, namely a
{\it selection-deletion} rule. This latter rule has been 
extensively used in algebraic combinatorics (see, for example, 
\cite{thi1} and references within). Let us also notice that 
for these type $II$ coproducts, explicit polynomial realizations 
have been recently obtained (see \cite{pol-real} and references within).
These polynomial realizations are particularly interesting 
in order to give new, straightforward, proofs of the coassociativity of the respective coproducts
(see, for example, J.-Y. Thibon's talk \cite{talk-jyt}).

It thus appears to us as a particularly interesting perspective for future work 
the investigation of the existence of connections
between these type $I$ and $II$ combinatorial Hopf algebra coproducts.
Such connections could eventually be obtained
by exhibiting explicit polynomial realizations for the 
type $I$ coproduct combinatorial Hopf algebras, thus completing the 
picture of combinatorial Hopf algebra polynomial realizations.



\noindent
{\small ${}^{a}${\it Centre de Physique Th\'eorique, Campus de Luminy,}} \\
{\small {\it Case 907 13288 Marseille cedex 9, France}} \\
{\small ${}^{b}${\it Laboratoire de Physique Th\'eorique, B\^atiment 210, Universit\'e Paris 11,}} \\
{\small {\it 91405 Orsay Cedex, France, EU}}\\
{\small ${}^{c}${\it 
Horia Hulubei National Institute for Physics and Nuclear Engineering,\\
P.O.B. MG-6, 077125 Magurele, Romania, EU}}\\
{\small ${}^{d}${\it Universit\'e Paris 13, Sorbonne Paris Cit\'e, 99, avenue Jean-Baptiste Cl\'ement \\
LIPN, Institut Galil\'ee, 
CNRS UMR 7030, F-93430, Villetaneuse, France, EU}}\\
\end{document}